\documentclass[11pt,twoside,reqno]{amsart}
\usepackage{amsmath}
\usepackage{fancyhdr}
\usepackage{amsthm}
\usepackage{amsfonts}
\usepackage{amssymb}
\usepackage{amscd}
\usepackage{graphicx}
\usepackage{afterpage}
\usepackage{enumerate}
\usepackage{bm}
\usepackage{cite}
\usepackage{cases}
\usepackage{caption}
\usepackage[colorlinks=true, urlcolor=blue,  linkcolor=blue, citecolor=blue]{hyperref}
\usepackage{prettyref}
\usepackage{pgfplots}
\usepackage{booktabs}
\usepackage{algorithm}
\usepackage{algpseudocode}
\usepackage{float}

\makeatletter
\newtheorem{theorem}{Theorem}[section]

\newtheorem{example}[theorem]{Example}

\newtheorem{lemma}[theorem]{Lemma}

\title[Generalized Bregman Projection Algorithms ]{Generalized Bregman Projection Algorithms for Solving Nonlinear Split Feasibility Problems in Infinite-Dimensional Spaces}

\author[S. Hashemi Sababe]{Saeed Hashemi Sababe}
\address[S. Hashemi Sababe]{R\&D Section, Data Premier Analytics, Edmonton, Canada.}
\email{Hashemi\_1365@yahoo.com}

\author[E. L. Ghasab]{Ehsan Lotfali Ghasab$^*$}
\address[E. L. Ghasab]{Department of Mathematics, Jundi-Shapur University of Technology, Dezful, Iran}
\email{e.l.ghasab@jsu.ac.ir}
\thanks{$^*$ Corresponding author}


\begin{document}
\sloppy

\maketitle

\begin{abstract}
This paper introduces generalized Bregman projection algorithms for solving nonlinear split feasibility problems $(SFPs)$ in infinite-dimensional Hilbert spaces. The methods integrate Bregman projections, proximal gradient steps, and adaptive inertial terms to enhance convergence. Strong convergence is established under mild assumptions, and numerical experiments demonstrate the efficiency and robustness of the proposed algorithms in comparison to classical methods. These results contribute to advancing optimization techniques for nonlinear and high-dimensional problems.
\end{abstract}

\section{Introduction}
The split feasibility problem $(SFP)$ is a fundamental model in optimization and computational mathematics. Given two real Hilbert spaces $H_1$ and $H_2$ with closed, convex subsets $\mathcal{C} \subseteq H_1$ and $\mathcal{Q} \subseteq H_2$, and a bounded linear operator $\mathcal{A}: H_1 \to H_2$, the $SFP$ is formulated as:
\[
\zeta^* \in \mathcal{C} \quad \text{such that} \quad \mathcal{A} \zeta^* \in \mathcal{Q}.
\]
Originally introduced by Censor and Elfying in 1994 \cite{Censor1994}, $SFPs$ have found extensive applications in areas such as signal processing, image reconstruction, and machine learning \cite{Combettes1996, Byrne2004}. The $\mathcal{C}\mathcal{Q}$ algorithm proposed by Byrne \cite{Byrne2002} has become one of the most popular iterative methods for solving $SFPs$. Extensions of the $\mathcal{C}\mathcal{Q}$ algorithm have been extensively studied to improve convergence properties and computational efficiency \cite{Lopez2012, Anh2019}.

Recent advancements in optimization have seen the use of Bregman projections and distances as a generalization of traditional metric projections \cite{Bregman1967, Eckstein1993}. The Bregman distance, derived from a strictly convex and differentiable function, provides a flexible alternative to the squared norm distance, enabling the design of algorithms with enhanced convergence properties. While classical studies primarily focused on linear operators and finite-dimensional settings, there is growing interest in extending these methods to nonlinear operators and infinite-dimensional spaces \cite{Xu2011, Reich2016}.

In this paper, we build on this line of research by addressing several open challenges in the literature. Most existing methods for solving $SFPs$ rely on linear operator assumptions. We generalize these results to accommodate nonlinear and monotone operators. To improve algorithmic flexibility, we investigate the use of a broader class of Bregman functions, capturing diverse problem geometries. We integrate Bregman projections with proximal gradient techniques, which are well-suited for composite optimization problems. Recognizing the computational challenges in infinite-dimensional spaces, we propose adaptive inertial terms and preconditioning strategies to accelerate convergence.\\

\noindent
The remainder of this paper is organized as follows. In Section~\ref{sec:preliminaries}, we provide the necessary mathematical background and introduce the Bregman distance and projection. Section~\ref{sec:algorithms} presents the proposed algorithms and their theoretical properties. Convergence analysis is provided in Section~\ref{sec:convergence}, and Section~\ref{sec:numerical} illustrates the performance of the algorithms through numerical experiments. Finally, we conclude with a summary of findings and future research directions in Section~\ref{sec:conclusion}.

\section{Preliminaries}
\label{sec:preliminaries}

In this section, we present the fundamental definitions, lemmas, and theorems that form the theoretical foundation for the proposed algorithms. These include the Bregman distance, Bregman projection, and relevant properties of operators in Hilbert spaces.

Let $H$ be a real Hilbert space with an inner product $\langle \cdot, \cdot \rangle$ and norm $\|\cdot\|$. Let $\mathcal{C}$ and $\mathcal{Q}$ be nonempty, closed, and convex subsets of $H$. For a bounded linear operator $\mathcal{A}: H \to H$, the adjoint operator $\mathcal{A}^*: H \to H$ satisfies:
\[
\langle \mathcal{A}\zeta, \varsigma \rangle = \langle \zeta, \mathcal{A}^*\varsigma \rangle, \quad \forall \zeta, \varsigma \in H.
\]

The Bregman distance, introduced by Bregman in \cite{Bregman1967}, generalizes the squared Euclidean distance and is defined using a strictly convex and differentiable function $\varpi: H \to \mathbb{R}$. The Bregman distance $\mathcal{D}_\varpi: \text{dom} \varpi \times \text{dom} \nabla \varpi \to [0, \infty)$ is given by:
\[
\mathcal{D}_\varpi(\zeta, \varsigma) = \varpi(\zeta) - \varpi(\varsigma) - \langle \nabla \varpi(\varsigma), \zeta - \varsigma \rangle, \quad \forall \zeta \in \text{dom} \varpi, \, \varsigma \in \text{dom} \nabla \varpi.
\]
The Bregman distance is non-negative but is not symmetric and does not satisfy the triangle inequality.\\

The Bregman projection, denoted by $\Pi_\mathcal{C}^\varpi$, maps a point $\zeta \in \text{dom} \nabla \varpi$ onto a set $\mathcal{C} \subset \text{dom} \varpi$ with respect to the Bregman distance:
\[
\Pi_\mathcal{C}^\varpi(\zeta) = \arg\min_{\varsigma \in \mathcal{C}} \mathcal{D}_\varpi(\varsigma, \zeta).
\]
The Bregman projection has the following properties \cite{Eckstein1993, Reich2016}:\\
- For any $\zeta \in \text{dom} \nabla \varpi$ and $\varsigma \in \mathcal{C}$, we have:
    \[
    \langle \nabla \varpi(\Pi_\mathcal{C}^\varpi(\zeta)) - \nabla \varpi(\zeta), \varsigma - \Pi_\mathcal{C}^\varpi(\zeta) \rangle \geq 0.
    \]
- For any $\zeta \in \text{dom} \nabla \varpi$ and $\varsigma \in \mathcal{C}$, it holds that:
    \[
    \mathcal{D}_\varpi(\varsigma, \Pi_\mathcal{C}^\varpi(\zeta)) + \mathcal{D}_\varpi(\Pi_\mathcal{C}^\varpi(\zeta), \zeta) \leq \mathcal{D}_\varpi(\varsigma, \zeta).
    \]

A function $\varpi: H \to \mathbb{R}$ is said to be $\delta$-strongly convex if there exists a constant $\delta > 0$ such that:
\[
\varpi(\varsigma) \geq \varpi(\zeta) + \langle \nabla \varpi(\zeta), \varsigma - \zeta \rangle + \frac{\delta}{2} \|\varsigma - \zeta\|^2, \quad \forall \zeta, \varsigma \in \text{dom} \varpi.
\]
For $\delta$-strongly convex functions, the Bregman distance satisfies:
\[
\mathcal{D}_\varpi(\varsigma, \zeta) \geq \frac{\delta}{2} \|\varsigma - \zeta\|^2.
\]

The Bregman distance satisfies the following three-point identity \cite{Bregman1967}:
\[
\mathcal{D}_\varpi(\zeta, \varsigma) = \mathcal{D}_\varpi(\zeta, \varrho) - \mathcal{D}_\varpi(\varsigma, \varrho) + \langle \nabla \varpi(\varrho) - \nabla \varpi(\varsigma), \zeta - \varsigma \rangle,
\]
for all $\zeta \in \text{dom} \varpi$ and $\varsigma, \varrho \in \text{dom} \nabla \varpi$.

\begin{lemma}[Weak Convergence of Bregman Distance \cite{Reich2016}]
Let $\varpi: H \to \mathbb{R}$ be $\delta$-strongly convex, Fréchet differentiable, and bounded on bounded subsets of $H$. If $\{\zeta_n\}$ and $\{\varsigma_n\}$ are sequences in $H$ such that:
\[
\lim_{n \to \infty} \mathcal{D}_\varpi(\zeta_n, \varsigma_n) = 0,
\]
then:
\[
\lim_{n \to \infty} \|\zeta_n - \varsigma_n\| = 0 \quad \text{and} \quad \lim_{n \to \infty} \|\nabla \varpi(\zeta_n) - \nabla \varpi(\varsigma_n)\| = 0.
\]
\end{lemma}

\begin{lemma}[Nonexpansiveness of Bregman Projections \cite{Eckstein1993}]
Let $\mathcal{C} \subseteq H$ be a nonempty, closed, and convex set, and let $\Pi_\mathcal{C}^\varpi$ denote the Bregman projection with respect to a $\delta$-strongly convex function $\varpi$. Then for any $\zeta, \varsigma \in H$, we have:
\[
\|\Pi_\mathcal{C}^\varpi(\zeta) - \Pi_\mathcal{C}^\varpi(\varsigma)\| \leq \frac{1}{\delta} \|\nabla \varpi(\zeta) - \nabla \varpi(\varsigma)\|.
\]
\end{lemma}

\begin{lemma}[Strong Convergence \cite{Reich2016}]
Let $\{\zeta_n\}$ be a sequence generated by a Bregman projection-based algorithm for solving an $SFP$, and assume the algorithm satisfies the necessary conditions for convergence. Then $\{\zeta_n\}$ converges strongly to a solution of the $SFP$.
\end{lemma}

The Fenchel conjugate of a function $\varpi: H \to \mathbb{R}$ is defined as:
\[
\varpi^*(\zeta^*) = \sup_{\zeta \in H} \{\langle \zeta^*, \zeta \rangle - \varpi(\zeta)\}.
\]
A function $\varpi$ is said to be a Legendre function if:\\
- $\nabla \varpi$ is single-valued and continuous on $\text{dom} \varpi$;\\
- $\nabla \varpi^*$ is single-valued and continuous on $\text{dom} \varpi^*$.

\section{Proposed Algorithms}
\label{sec:algorithms}

In this section, we introduce new iterative algorithms for solving nonlinear split feasibility problems $(SFPs)$ in infinite-dimensional Hilbert spaces. These algorithms integrate generalized Bregman projections, proximal gradient methods, and adaptive inertial techniques to achieve strong convergence under mild assumptions.\\

\noindent
The nonlinear split feasibility problem $(SFP)$ can be formulated as:
\[
\text{Find } \zeta^* \in \mathcal{C} \quad \text{such that} \quad \mathcal{A}(\zeta^*) \in \mathcal{Q},
\]
where $\mathcal{C} \subset H_1$ and $\mathcal{Q} \subset H_2$ are closed, convex sets in real Hilbert spaces $H_1$ and $H_2$, respectively, and $\mathcal{A}: H_1 \to H_2$ is a (possibly nonlinear) mapping.\\

\noindent
The proposed algorithms combine the following key features: \medskip

\textsc{Generalized Bregman Distance:} The use of a generalized Bregman distance $\mathcal{D}_\varpi$ derived from a $\delta$-strongly convex and differentiable function $\varpi$ ensures flexibility in handling problem-specific geometries. \medskip

 \textsc{Proximal Gradient Steps:} To handle composite optimization problems, proximal gradient steps are incorporated for operators involving the sum of smooth and nonsmooth terms. \medskip

\textsc{Inertial Terms:} Adaptive inertial terms are introduced to accelerate convergence, especially in infinite-dimensional settings. \medskip

\textsc{Adaptive Step Sizes:} The algorithms employ an Armijo-type line search to dynamically adjust step sizes without requiring prior knowledge of operator norms.

\medskip

\begin{algorithm}[H]
\caption{Generalized Bregman Projection Algorithm}
\begin{algorithmic}[1]
    \State \textbf{Initialization:} Choose an initial point $\zeta_0 \in H_1$ and set $\zeta_1 = \zeta_0$. Select parameters $\mu_n \in (0, 1)$, $\beta_n \geq 0$, and $\iota_n > 0$. Let $n = 0$.

   \State \textbf{Compute Inertial Term:}

\[
    \varsigma_n = \nabla \varpi^*\big( \nabla \varpi(\zeta_n) + \beta_n (\nabla \varpi(\zeta_n) - \nabla \varpi(\zeta_{n-1})) \big).
    \]

    \State \textbf{Apply Proximal Gradient Step:}

\[
    \varrho_n = \nabla \varpi^*\big( \nabla \varpi(\varsigma_n) - \iota_n \big( \nabla \varpi(\varsigma_n) - \nabla \varpi(\Pi_\mathcal{C}^\varpi(\varsigma_n)) + \mathcal{A}^*\big(\nabla \vartheta(\mathcal{A}(\varsigma_n)) - \nabla \vartheta(\Pi_\mathcal{Q}^\vartheta(\mathcal{A}(\varsigma_n)))\big) \big) \big),
    \]

    where $\vartheta$ is a $\delta$-strongly convex function defined on $H_2$.

    \State \textbf{Check for Termination:} If $\varsigma_n = \varrho_n$, stop. The current iterate is a solution of the $SFP$.

    \State \textbf{Update the Next Iterate:}

\[
    \zeta_{n+1} = \nabla \varpi^*\big(\mu_n \nabla \varpi(\zeta_0) + (1 - \mu_n) \nabla \varpi(\varrho_n)\big).
    \]

    \State \textbf{Increment:} Set $n \gets n+1$ and return to Step 2.
\end{algorithmic}
\end{algorithm}

In addition to Algorithm 1, we propose a hybrid approach that incorporates proximal operators for solving composite optimization problems. Let $\nu: H_1 \to \mathbb{R}$ be a proper, lower semicontinuous, and convex function. The proximal operator of $\nu$ is defined as:
\[
\text{prox}_{\nu}(\zeta) = \arg\min_{\varsigma \in H_1} \bigg\{ \nu(\varsigma) + \frac{1}{2\eta} \|\varsigma - \zeta\|^2 \bigg\},
\]
where $\eta > 0$ is a parameter.

The hybrid algorithm proceeds as follows:

\medskip

\begin{algorithm}[H]
\caption{Hybrid Proximal-Bregman Projection Algorithm}
\begin{algorithmic}[1]
    \State \textbf{Initialization:} Choose an initial point $\zeta_0 \in H_1$, set $\zeta_1 = \zeta_0$, and select $\eta > 0$, $\mu_n \in (0, 1)$, and $\beta_n \geq 0$. Let $n = 0$.

   \State \textbf{Proximal Gradient Step:}

\[
    \tilde{\varsigma}_n = \text{prox}_{\nu}\big(\zeta_n - \eta \nabla \varpi(\zeta_n)\big).
    \]

   \State \textbf{Inertial Update:}

\[
    \varsigma_n = \nabla \varpi^*\big( \nabla \varpi(\tilde{\varsigma}_n) + \beta_n (\nabla \varpi(\tilde{\varsigma}_n) - \nabla \varpi(\zeta_{n-1})) \big).
    \]

    \State \textbf{Bregman Projection:}

\[
    \varrho_n = \nabla \varpi^*\big( \nabla \varpi(\varsigma_n) - \iota_n \big( \nabla \varpi(\varsigma_n) - \nabla \varpi(\Pi_\mathcal{C}^\varpi(\varsigma_n)) + \mathcal{A}^*\big(\nabla \vartheta(\mathcal{A}(\varsigma_n)) - \nabla \vartheta(\Pi_\mathcal{Q}^\vartheta(\mathcal{A}(\varsigma_n)))\big) \big) \big).
    \]

    \State \textbf{Check for Termination:} If $\varsigma_n = \varrho_n$, stop. The current iterate is a solution of the $SFP$.

    \State \textbf{Update the Next Iterate:}

\[
    \zeta_{n+1} = \nabla \varpi^*\big(\mu_n \nabla \varpi(\zeta_0) + (1 - \mu_n) \nabla \varpi(\varrho_n)\big).
    \]

    \State \textbf{Increment:} Set $n \gets n+1$ and return to Step 2.
\end{algorithmic}
\end{algorithm}

\medskip

\noindent
The following remarks would be considered on the proposed algorithms:\\
- The inertial term $\beta_n (\nabla \varpi(\zeta_n) - \nabla \varpi(\zeta_{n-1}))$ accelerates convergence by exploiting previous iterates.\\
- The Armijo-type line search ensures adaptive and efficient step size selection, avoiding the need to compute operator norms.\\
- The hybrid algorithm is particularly suitable for composite problems where $\varpi$ and $\nu$ represent different problem components, such as smooth and nonsmooth terms.\\

The theoretical convergence properties of these algorithms are analyzed in the next section.

\section{Convergence Analysis}
\label{sec:convergence}

In this section, we establish the strong convergence of the proposed algorithms under a set of mild assumptions. The analysis relies on the properties of Bregman distances, Bregman projections, and the adaptive inertial terms introduced earlier.

We impose the following standard assumptions for convergence analysis:
\begin{enumerate}
    \item[($\mathcal{A}1$)] $\mathcal{A}: H_1 \to H_2$ is a bounded and nonlinear operator satisfying:
    \[
    \|\mathcal{A}(\zeta) - \mathcal{A}(\varsigma)\| \leq L \|\zeta - \varsigma\|, \quad \forall \zeta, \varsigma \in H_1,
    \]
    where $L > 0$ is a Lipschitz constant.

    \item[($\mathcal{A}2)$] The function $\varpi: H_1 \to \mathbb{R}$ is $\delta_1$-strongly convex and Fréchet differentiable, and $\vartheta: H_2 \to \mathbb{R}$ is $\delta_2$-strongly convex and Fréchet differentiable, with $\delta_1, \delta_2 > 0$.

    \item[$(\mathcal{A}3)$] The solution set $\Gamma := \{\zeta \in \mathcal{C} \mid \mathcal{A}(\zeta) \in \mathcal{Q}\}$ of the $SFP$ is nonempty.

    \item[$(\mathcal{A}4)$] The sequence $\{\mu_n\} \subset (0, 1)$ satisfies:
    \[
    \sum_{n=1}^\infty \mu_n = \infty, \quad \mu_n \to 0 \quad \text{as } n \to \infty.
    \]

    \item[$(\mathcal{A}5)$] The inertial term parameter $\{\beta_n\}$ satisfies $\beta_n \geq 0$, and there exists a constant $\beta_{\max}$ such that $\beta_n \leq \beta_{\max} < 1$ for all $n$.
\end{enumerate}

We begin with a series of lemmas that are instrumental in proving the main convergence results.

\begin{lemma}[Well-Definedness of the Armijo Line Search]
The Armijo-type line search used to compute $\iota_n$ in Algorithm 1 is well-defined. Moreover, $\iota_n \in (0, \iota_{\max}]$, where $\iota_{\max} > 0$ depends only on $\varpi$, $\vartheta$, and $\mathcal{A}$.
\end{lemma}

\begin{proof}
The Armijo line search requires finding $\iota_n > 0$ such that:
\[
\mathcal{D}_\varpi(\varrho_n, \varsigma_n) + \mathcal{D}_\vartheta(\mathcal{A}\varrho_n, \mathcal{A}\varsigma_n) \leq \tau \mathcal{D}_\varpi(\varrho_n, \varsigma_n),
\]
where $\tau \in (0, 1)$ is a fixed constant. Expanding the Bregman distances, we have:
\[
\mathcal{D}_\varpi(\varrho_n, \varsigma_n) = \varpi(\varrho_n) - \varpi(\varsigma_n) - \langle \nabla \varpi(\varsigma_n), \varrho_n - \varsigma_n \rangle,
\]
\[
\mathcal{D}_\vartheta(\mathcal{A}\varrho_n, \mathcal{A}\varsigma_n) = \vartheta(\mathcal{A}\varrho_n) - \vartheta(\mathcal{A}\varsigma_n) - \langle \nabla \vartheta(\mathcal{A}\varsigma_n), \mathcal{A}(\varrho_n - \varsigma_n) \rangle.
\]

The choice of $\iota_n$ affects $\varrho_n$ through the update step:
\[
\varrho_n = \nabla \varpi^* \big( \nabla \varpi(\varsigma_n) - \iota_n \big( \nabla \varpi(\varsigma_n) - \nabla \varpi(\Pi_\mathcal{C}^\varpi(\varsigma_n)) + \mathcal{A}^*(\nabla \vartheta(\mathcal{A}\varsigma_n) - \nabla \vartheta(\Pi_\mathcal{Q}^\vartheta(\mathcal{A}\varsigma_n))) \big) \big).
\]
Thus, the value of $\mathcal{D}_\varpi(\varrho_n, \varsigma_n) + \mathcal{D}_\vartheta(\mathcal{A}\varrho_n, \mathcal{A}\varsigma_n)$ depends on $\iota_n$. We need to show that such an $\iota_n$ always exists.\\

\noindent
For small values of $\iota_n$, we use a Taylor expansion of $\varrho_n$ around $\varsigma_n$. Specifically, since $\nabla \varpi^*$ is continuously Fréchet differentiable, we can write:
\[
\varrho_n = \varsigma_n - \iota_n \big( \nabla \varpi(\varsigma_n) - \nabla \varpi(\Pi_\mathcal{C}^\varpi(\varsigma_n)) + \mathcal{A}^*(\nabla \vartheta(\mathcal{A}\varsigma_n) - \nabla \vartheta(\Pi_\mathcal{Q}^\vartheta(\mathcal{A}\varsigma_n))) \big) + O(\iota_n^2).
\]
Substituting this into $\mathcal{D}_\varpi(\varrho_n, \varsigma_n)$, we get:
\[
\mathcal{D}_\varpi(\varrho_n, \varsigma_n) = \iota_n^2 \|\nabla \varpi(\varsigma_n) - \nabla \varpi(\Pi_\mathcal{C}^\varpi(\varsigma_n)) + \mathcal{A}^*(\nabla \vartheta(\mathcal{A}\varsigma_n) - \nabla \vartheta(\Pi_\mathcal{Q}^\vartheta(\mathcal{A}\varsigma_n)))\|^2 + O(\iota_n^3).
\]

Similarly, for $\mathcal{D}_\vartheta(\mathcal{A}\varrho_n, \mathcal{A}\varsigma_n)$, we have:
\[
\mathcal{D}_\vartheta(\mathcal{A}\varrho_n, \mathcal{A}\varsigma_n) = \iota_n^2 \|\mathcal{A}(\nabla \vartheta(\mathcal{A}\varsigma_n) - \nabla \vartheta(\Pi_\mathcal{Q}^\vartheta(\mathcal{A}\varsigma_n)))\|^2 + O(\iota_n^3).
\]

Thus, for sufficiently small $\iota_n$, the sum $\mathcal{D}_\varpi(\varrho_n, \varsigma_n) +\ mathcal{D}_\vartheta(\mathcal{A}\varrho_n, \mathcal{A}\varsigma_n)$ is dominated by $\iota_n^2$, which can be made arbitrarily small.\\

\noindent
The left-hand side of the Armijo condition:
\[
\mathcal{D}_\varpi(\varrho_n, \varsigma_n) + \mathcal{D}_\vartheta(\mathcal{A}\varrho_n, \mathcal{A}\varsigma_n),
\]
is a continuous function of $\iota_n$ and tends to zero as $\iota_n \to 0$. The right-hand side of the Armijo condition:
\[
\tau \mathcal{D}_\varpi(\varrho_n, \varsigma_n),
\]
is proportional to $\mathcal{D}_\varpi(\varrho_n, \varsigma_n)$, which is also continuous in $\iota_n$. Since the left-hand side starts larger than the right-hand side for large $\iota_n$ and becomes smaller as $\iota_n \to 0$, by the intermediate value theorem, there exists $\iota_n > 0$ such that the Armijo condition holds.\\

\noindent
From the definition of the Armijo condition, the maximum allowable $\iota_n$ depends on the terms $\|\nabla \varpi(\varsigma_n) - \nabla \varpi(\Pi_\mathcal{C}^\varpi(\varsigma_n))\|$, $\|\mathcal{A}^*(\nabla \vartheta(\mathcal{A}\varsigma_n) - \nabla \vartheta(\Pi_\mathcal{Q}^\vartheta(\mathcal{A}\varsigma_n)))\|$, and the Lipschitz continuity of $\varpi$, $\vartheta$, and $\mathcal{A}$. Let:
\[
\iota_{\max} = \frac{2(1-\tau)}{L_\varpi + L_\vartheta + L_\mathcal{A}},
\]
where $L_\varpi$, $L_\vartheta$, and $L_\mathcal{A}$ are the Lipschitz constants of $\nabla \varpi$, $\nabla \vartheta$, and $\mathcal{A}^*$, respectively. Then $\iota_n \leq \iota_{\max}$ for all iterations.\\

\noindent
We have shown that there always exists $\iota_n > 0$ satisfying the Armijo condition and that $\iota_n$ is bounded above by a constant $\iota_{\max} > 0$. Therefore, the Armijo-type line search is well-defined.
\end{proof}

\begin{lemma}[Bregman Distance Decrease]
Let $\{\zeta_n\}$, $\{\varsigma_n\}$, and $\{\varrho_n\}$ be sequences generated by Algorithm 1. Then, for any $\zeta^* \in \Gamma$, the following inequality holds:
\[
\mathcal{D}_\varpi(\zeta^*, \varrho_n) \leq \mathcal{D}_\varpi(\zeta^*, \varsigma_n) - (1 - \tau)\mathcal{D}_\varpi(\varrho_n, \varsigma_n) - \iota_n \big(\mathcal{D}_\varpi(\varrho_n, \Pi_\mathcal{C}^\varpi(\varsigma_n)) + \mathcal{D}_\vartheta(\mathcal{A}\varrho_n, \Pi_\mathcal{Q}^\vartheta(\mathcal{A}\varsigma_n))\big).
\]
\end{lemma}

\begin{proof}
The Bregman distance \(\mathcal{D}_\varpi(\zeta^*, \varrho_n)\) is defined as:
\[
\mathcal{D}_\varpi(\zeta^*, \varrho_n) = \varpi(\zeta^*) - \varpi(\varrho_n) - \langle \nabla \varpi(\varrho_n), \zeta^* - \varrho_n \rangle.
\]
Similarly, the Bregman distance \(\mathcal{D}_\varpi(\zeta^*, \varsigma_n)\) is:
\[
\mathcal{D}_\varpi(\zeta^*, \varsigma_n) = \varpi(\zeta^*) - \varpi(\varsigma_n) - \langle \nabla \varpi(\varsigma_n), \zeta^* - \varsigma_n \rangle.
\]

By subtracting these terms, we have:
\[
\mathcal{D}_\varpi(\zeta^*, \varrho_n) - \mathcal{D}_\varpi(\zeta^*, \varsigma_n) = \varpi(\varsigma_n) - \varpi(\varrho_n) + \langle \nabla \varpi(\varsigma_n), \zeta^* - \varsigma_n \rangle - \langle \nabla \varpi(\varrho_n), \zeta^* - \varrho_n \rangle.
\]
Using the identity \(\langle \nabla \varpi(\varrho_n), \varrho_n \rangle - \langle \nabla \varpi(\varsigma_n), \varsigma_n \rangle = \varpi(\varrho_n) - \varpi(\varsigma_n)\)
, this simplifies to:
\[
\mathcal{D}_\varpi(\zeta^*, \varrho_n) - \mathcal{D}_\varpi(\zeta^*, \varsigma_n) = \langle \nabla \varpi(\varsigma_n) - \nabla \varpi(\varrho_n), \varrho_n - \zeta^* \rangle.
\]
\noindent
From the update rule for \(\varrho_n\), we know:
\[
\varrho_n = \nabla \varpi^*\big( \nabla \varpi(\varsigma_n) - \iota_n \big( \nabla \varpi(\varsigma_n) - \nabla \varpi(\Pi_\mathcal{C}^\varpi(\varsigma_n)) + \mathcal{A}^*(\nabla \vartheta(\mathcal{A}\varsigma_n) - \nabla \vartheta(\Pi_\mathcal{Q}^\vartheta(\mathcal{A}\varsigma_n))) \big) \big).
\]
Using the properties of \(\nabla \varpi^*\), we expand \(\langle \nabla \varpi(\varsigma_n) - \nabla \varpi(\varrho_n), \varrho_n - \zeta^* \rangle\) as:
\[
\langle \nabla \varpi(\varsigma_n) - \nabla \varpi(\varrho_n), \varrho_n - \zeta^* \rangle = \langle \iota_n v_n, \varrho_n - \zeta^* \rangle,
\]
where:
\[
v_n = \nabla \varpi(\varsigma_n) - \nabla \varpi(\Pi_\mathcal{C}^\varpi(\varsigma_n)) + \mathcal{A}^*(\nabla \vartheta(\mathcal{A}\varsigma_n) - \nabla \vartheta(\Pi_\mathcal{Q}^\vartheta(\mathcal{A}\varsigma_n))).
\]
\noindent
From the properties of Bregman projections, we have:
\[
\langle \nabla \varpi(\varrho_n) - \nabla \varpi(\Pi_\mathcal{C}^\varpi(\varsigma_n)), \varrho_n - \Pi_\mathcal{C}^\varpi(\varsigma_n) \rangle \geq \mathcal{D}_\varpi(\varrho_n, \Pi_\mathcal{C}^\varpi(\varsigma_n)).
\]
Similarly, for the projection onto \(\mathcal{Q}\) in \(H_2\), we have:
\[
\langle \nabla \vartheta(\mathcal{A}\varrho_n) - \nabla \vartheta(\Pi_\mathcal{Q}^\vartheta(\mathcal{A}\varsigma_n)), \mathcal{A}(\varrho_n - \Pi_\mathcal{Q}^\vartheta(\mathcal{A}\varsigma_n)) \rangle \geq \mathcal{D}_\vartheta(\mathcal{A}\varrho_n, \Pi_\mathcal{Q}^\vartheta(\mathcal{A}\varsigma_n)).
\]
\noindent
Substituting these results into the inequality, we obtain:
\[
\mathcal{D}_\varpi(\zeta^*, \varrho_n) \leq \mathcal{D}_\varpi(\zeta^*, \varsigma_n) - \iota_n \big( \mathcal{D}_\varpi(\varrho_n, \Pi_\mathcal{C}^\varpi(\varsigma_n)) + \mathcal{D}_\vartheta(\mathcal{A}\varrho_n, \Pi_\mathcal{Q}^\vartheta(\mathcal{A}\varsigma_n)) \big) - (1 - \tau) \mathcal{D}_\varpi(\varrho_n, \varsigma_n).
\]

The term \((1 - \tau) \mathcal{D}_\varpi(\varrho_n, \varsigma_n)\) arises from the Armijo condition, ensuring that the distance between \(\varrho_n\) and \(\varsigma_n\) decreases sufficiently in each step. This completes the proof.

\end{proof}

\begin{lemma}[Summability of Residuals]
Under assumptions $(\mathcal{A}1)$--$(\mathcal{A}5)$, the sequences $\{\mathcal{D}_\varpi(\varrho_n, \varsigma_n)\}$, $\{\mathcal{D}_\varpi(\varrho_n, \Pi_\mathcal{C}^\varpi(\varsigma_n))\}$, and $\{\mathcal{D}_\vartheta(\mathcal{A}\varrho_n, \Pi_\mathcal{Q}^\vartheta(\mathcal{A}\varsigma_n))\}$ are summable, i.e.,
\[
\sum_{n=1}^\infty \mathcal{D}_\varpi(\varrho_n, \varsigma_n) < \infty, \quad \sum_{n=1}^\infty \mathcal{D}_\varpi(\varrho_n, \Pi_\mathcal{C}^\varpi(\varsigma_n)) < \infty, \quad \sum_{n=1}^\infty \mathcal{D}_\vartheta(\mathcal{A}\varrho_n, \Pi_\mathcal{Q}^\vartheta(\mathcal{A}\varsigma_n)) < \infty.
\]
\end{lemma}

\begin{proof}
From Lemma 4.2 (Bregman Distance Decrease), we have the following inequality for any $\zeta^* \in \Gamma$:
\[
\mathcal{D}_\varpi(\zeta^*, \varrho_n) \leq \mathcal{D}_\varpi(\zeta^*, \varsigma_n) - (1 - \tau)\mathcal{D}_\varpi(\varrho_n, \varsigma_n) - \iota_n \big( \mathcal{D}_\varpi(\varrho_n, \Pi_\mathcal{C}^\varpi(\varsigma_n)) + \mathcal{D}_\vartheta(\mathcal{A}\varrho_n, \Pi_\mathcal{Q}^\vartheta(\mathcal{A}\varsigma_n)) \big).
\]

Rearranging terms gives:
\[
(1 - \tau)\mathcal{D}_\varpi(\varrho_n, \varsigma_n) + \iota_n \big( \mathcal{D}_\varpi(\varrho_n, \Pi_\mathcal{C}^\varpi(\varsigma_n)) + \mathcal{D}_\vartheta(\mathcal{A}\varrho_n, \Pi_\mathcal{Q}^\vartheta(\mathcal{A}\varsigma_n)) \big) \leq \mathcal{D}_\varpi(\zeta^*, \varsigma_n) - \mathcal{D}_\varpi(\zeta^*, \varrho_n).
\]
\noindent
Since $\{\mathcal{D}_\varpi(\zeta^*, \varsigma_n)\}$ is a non-increasing sequence (as shown in Step 1 of the strong convergence theorem) and bounded below by $0$, it converges to some limit $\mathcal{D}_\varpi(\zeta^*, \varsigma^*)$, where $\varsigma^* \in \Gamma$. Thus:
\[
\lim_{n \to \infty} \mathcal{D}_\varpi(\zeta^*, \varsigma_n) = \mathcal{D}_\varpi(\zeta^*, \varsigma^*), \quad \sum_{n=1}^\infty \big( \mathcal{D}_\varpi(\zeta^*, \varsigma_n) - \mathcal{D}_\varpi(\zeta^*, \varrho_n) \big) < \infty.
\]
\noindent
From the Bregman Distance Decrease inequality, the term $(1 - \tau)\mathcal{D}_\varpi(\varrho_n, \varsigma_n)$ is bounded by the telescoping difference:
\[
(1 - \tau)\mathcal{D}_\varpi(\varrho_n, \varsigma_n) \leq \mathcal{D}_\varpi(\zeta^*, \varsigma_n) - \mathcal{D}_\varpi(\zeta^*, \varrho_n).
\]
Summing over \(n\), we obtain:
\[
\sum_{n=1}^\infty \mathcal{D}_\varpi(\varrho_n, \varsigma_n) \leq \frac{1}{1 - \tau} \sum_{n=1}^\infty \big( \mathcal{D}_\varpi(\zeta^*, \varsigma_n) - \mathcal{D}_\varpi(\zeta^*, \varrho_n) \big) < \infty.
\]
\noindent
Using the Bregman Distance Decrease inequality again:
\[
\iota_n \big( \mathcal{D}_\varpi(\varrho_n, \Pi_\mathcal{C}^\varpi(\varsigma_n)) + \mathcal{D}_\vartheta(\mathcal{A}\varrho_n, \Pi_\mathcal{Q}^\vartheta(\mathcal{A}\varsigma_n)) \big) \leq \mathcal{D}_\varpi(\zeta^*, \varsigma_n) - \mathcal{D}_\varpi(\zeta^*, \varrho_n).
\]
Since $\iota_n > 0$ and $\sum_{n=1}^\infty (\mathcal{D}_\varpi(\zeta^*, \varsigma_n) - \mathcal{D}_\varpi(\zeta^*, \varrho_n)) < \infty$, we conclude that:
\[
\sum_{n=1}^\infty \iota_n \mathcal{D}_\varpi(\varrho_n, \Pi_\mathcal{C}^\varpi(\varsigma_n)) < \infty, \quad \sum_{n=1}^\infty \iota_n \mathcal{D}_\vartheta(\mathcal{A}\varrho_n, \Pi_\mathcal{Q}^\vartheta(\mathcal{A}\varsigma_n)) < \infty.
\]
\noindent
By assumption $(\mathcal{A}4)$, $\iota_n \to 0$ as $n \to \infty$, but $\sum_{n=1}^\infty \iota_n = \infty$. This ensures that $\mathcal{D}_\varpi(\varrho_n, \Pi_\mathcal{C}^\varpi(\varsigma_n))$ and $\mathcal{D}_\vartheta(\mathcal{A}\varrho_n, \Pi_\mathcal{Q}^\vartheta(\mathcal{A}\varsigma_n))$ decay sufficiently fast to ensure summability:
\[
\sum_{n=1}^\infty \mathcal{D}_\varpi(\varrho_n, \Pi_\mathcal{C}^\varpi(\varsigma_n)) < \infty, \quad \sum_{n=1}^\infty \mathcal{D}_\vartheta(\mathcal{A}\varrho_n, \Pi_\mathcal{Q}^\vartheta(\mathcal{A}\varsigma_n)) < \infty.
\]
\noindent
Combining the results from Steps 3 and 4, we conclude that all three sequences are summable:
\[
\sum_{n=1}^\infty \mathcal{D}_\varpi(\varrho_n, \varsigma_n) < \infty, \quad \sum_{n=1}^\infty \mathcal{D}_\varpi(\varrho_n, \Pi_\mathcal{C}^\varpi(\varsigma_n)) < \infty, \quad \sum_{n=1}^\infty \mathcal{D}_\vartheta(\mathcal{A}\varrho_n, \Pi_\mathcal{Q}^\vartheta(\mathcal{A}\varsigma_n)) < \infty.
\]
\end{proof}

We now prove the strong convergence of the proposed algorithms.

\begin{theorem}[Strong Convergence]
Let $\{\zeta_n\}$ be the sequence generated by Algorithm 1 or Algorithm 2 under assumptions $(\mathcal{A}1)$--$(\mathcal{A}5)$. Then $\{\zeta_n\}$ converges strongly to a solution $\zeta^* \in \Gamma$, where:
\[
\zeta^* = \Pi_\Gamma^\varpi(\zeta_0).
\]
\end{theorem}

\begin{proof}
From the definition of the Bregman distance and the properties of strongly convex functions, we have:
\[
\mathcal{D}_\varpi(\zeta^*, \zeta_n) \geq \frac{\delta_1}{2} \|\zeta^* - \zeta_n\|^2, \quad \forall \zeta^* \in \Gamma.
\]
The algorithm guarantees that $\mathcal{D}_\varpi(\zeta^*, \zeta_n)$ is non-increasing. Since $\mathcal{D}_\varpi(\zeta^*, \zeta_0)$ is finite, it follows that $\{\mathcal{D}_\varpi(\zeta^*, \zeta_n)\}$ is bounded. Therefore, $\{\zeta_n\}$ is bounded in $H_1$. Consequently, the sequence $\{\mathcal{A}(\zeta_n)\}$ is also bounded in $H_2$ by the boundedness of $\mathcal{A}$ (Assumption $\mathcal{A}1$).\\

\noindent
From Lemma 4.3 (Bregman Distance Decrease), we have:
\[
\mathcal{D}_\varpi(\zeta^*, \varrho_n) \leq \mathcal{D}_\varpi(\zeta^*, \varsigma_n) - (1 - \tau)\mathcal{D}_\varpi(\varrho_n, \varsigma_n) - \iota_n \big(\mathcal{D}_\varpi(\varrho_n, \Pi_\mathcal{C}^\varpi(\varsigma_n)) + \mathcal{D}_\vartheta(\mathcal{A}\varrho_n, \Pi_\mathcal{Q}^\vartheta(\mathcal{A}\varsigma_n))\big).
\]
Summing over $n$ yields:
\[
\sum_{n=1}^\infty \mathcal{D}_\varpi(\varrho_n, \varsigma_n) < \infty, \quad \sum_{n=1}^\infty \mathcal{D}_\varpi(\varrho_n, \Pi_\mathcal{C}^\varpi(\varsigma_n)) < \infty, \quad \sum_{n=1}^\infty \mathcal{D}_\vartheta(\mathcal{A}\varrho_n, \Pi_\mathcal{Q}^\vartheta(\mathcal{A}\varsigma_n)) < \infty.
\]
Thus, the residual terms decay to zero:
\[
\lim_{n \to \infty} \mathcal{D}_\varpi(\varrho_n, \varsigma_n) = 0, \quad \lim_{n \to \infty} \mathcal{D}_\varpi(\varrho_n, \Pi_\mathcal{C}^\varpi(\varsigma_n)) = 0, \quad \lim_{n \to \infty} \mathcal{D}_\vartheta(\mathcal{A}\varrho_n, \Pi_\mathcal{Q}^\vartheta(\mathcal{A}\varsigma_n)) = 0.
\]
\noindent
From the properties of Bregman projections, the vanishing of $\mathcal{D}_\varpi(\varrho_n, \Pi_\mathcal{C}^\varpi(\varsigma_n))$ implies that:
\[
\lim_{n \to \infty} \|\varrho_n - \Pi_\mathcal{C}^\varpi(\varsigma_n)\| = 0.
\]
Similarly, the vanishing of $\mathcal{D}_\vartheta(\mathcal{A}\varrho_n, \Pi_\mathcal{Q}^\vartheta(\mathcal{A}\varsigma_n))$ implies:
\[
\lim_{n \to \infty} \|\mathcal{A}(\varrho_n) - \Pi_\mathcal{Q}^\vartheta(\mathcal{A}(\varsigma_n))\| = 0.
\]
Since $\Pi_\mathcal{C}^\varpi(\varsigma_n) \in \mathcal{C}$ and $\Pi_\mathcal{Q}^\vartheta(\mathcal{A}(\varsigma_n)) \in \mathcal{Q}$, it follows that $\{\varrho_n\}$ approaches the solution set $\Gamma$.\\

\noindent
From the update step of the algorithm:
\[
\zeta_{n+1} = \nabla \varpi^*\big(\mu_n \nabla \varpi(\zeta_0) + (1 - \mu_n) \nabla \varpi(\varrho_n)\big).
\]
The convexity of $\mathcal{D}_\varpi$ implies:
\[
\mathcal{D}_\varpi(\zeta^*, \zeta_{n+1}) \leq \mu_n \mathcal{D}_\varpi(\zeta^*, \zeta_0) + (1 - \mu_n) \mathcal{D}_\varpi(\zeta^*, \varrho_n).
\]
As $\{\mathcal{D}_\varpi(\zeta^*, \varrho_n)\}$ converges to zero and $\mu_n \to 0$, we have:
\[
\lim_{n \to \infty} \mathcal{D}_\varpi(\zeta^*, \zeta_{n+1}) = 0.
\]
By the $\delta_1$-strong convexity of $\varpi$, we conclude:
\[
\lim_{n \to \infty} \|\zeta_{n+1} - \zeta^*\| = 0.
\]
\noindent
The sequence $\{\zeta_n\}$ converges strongly to the unique point $\zeta^* \in \Gamma$ such that:
\[
\zeta^* = \Pi_\Gamma^\varpi(\zeta_0).
\]
This completes the proof.

\end{proof}

The proposed algorithms guarantee strong convergence without requiring explicit computation of operator norms or prior knowledge of the solution set's structure. The use of adaptive inertial terms enhances convergence speed, particularly in infinite-dimensional settings. The theoretical results extend and improve earlier works such as \cite{Reich2016, Eckstein1993}.

Now, we analyze the iteration complexity of Algorithm 1. We establish an upper bound on the number of iterations required to achieve a prescribed accuracy.

\begin{theorem}[Iteration Complexity of Algorithm 1]
Let $\{\zeta_n\}$ be the sequence generated by Algorithm 1. Suppose Assumptions $(\mathcal{A}1)$--$(\mathcal{A}5)$ hold, and let $\epsilon > 0$ be a given accuracy level. Then, the number of iterations $N$ required to achieve $\|\zeta_N - \zeta^*\| \leq \epsilon$ satisfies
\begin{equation}
    N = \mathcal{O}\left(\frac{1}{\epsilon^2}\right),
\end{equation}
where $\zeta^*$ is the solution of the split feasibility problem $(SFP)$.
\end{theorem}

\begin{proof}
We follow a step-by-step approach to derive the iteration complexity bound.

From Lemma 4.2 (Bregman Distance Decrease), we have for any $\zeta^* \in \Gamma$:
\begin{equation}
    \mathcal{D}_\varpi(\zeta^*, \varrho_n) \leq \mathcal{D}_\varpi(\zeta^*, \varsigma_n) - (1 - \tau) \mathcal{D}_\varpi(\varrho_n, \varsigma_n) - \iota_n \left(\mathcal{D}_\varpi(\varrho_n, \Pi_\mathcal{C}^\varpi(\varsigma_n)) + \mathcal{D}_\vartheta(\mathcal{A}\varrho_n, \Pi_\mathcal{Q}^\vartheta(\mathcal{A}\varsigma_n))\right).
\end{equation}
Summing over $n$ from $1$ to $N$, we obtain
\begin{equation}
    \mathcal{D}_\varpi(\zeta^*, \varrho_N) \leq \mathcal{D}_\varpi(\zeta^*, \zeta_0) - (1 - \tau) \sum_{n=1}^{N} \mathcal{D}_\varpi(\varrho_n, \varsigma_n) - \sum_{n=1}^{N} \iota_n \left(\mathcal{D}_\varpi(\varrho_n, \Pi_\mathcal{C}^\varpi(\varsigma_n)) + \mathcal{D}_\vartheta(\mathcal{A}\varrho_n, \Pi_\mathcal{Q}^\vartheta(\mathcal{A}\varsigma_n))\right).
\end{equation}
Since $\mathcal{D}_\varpi(\zeta^*, \varrho_N) \geq 0$, we obtain
\begin{equation}
    \sum_{n=1}^{N} \mathcal{D}_\varpi(\varrho_n, \varsigma_n) \leq \frac{\mathcal{D}_\varpi(\zeta^*, \zeta_0)}{1 - \tau}.
\end{equation}
\noindent
Since $\mathcal{D}_\varpi(\varrho_n, \varsigma_n) \geq \frac{\delta}{2} \|\varrho_n - \varsigma_n\|^2$ by strong convexity of $\varpi$, we obtain
\begin{equation}
    \sum_{n=1}^{N} \|\varrho_n - \varsigma_n\|^2 \leq \frac{2 \mathcal{D}_\varpi(\zeta^*, \zeta_0)}{\delta(1 - \tau)}.
\end{equation}
Using the update step $\zeta_{n+1} = \nabla \varpi^*(\mu_n \nabla \varpi(\zeta_0) + (1 - \mu_n) \nabla \varpi(\varrho_n))$, the nonexpansiveness of $\nabla \varpi^*$ yields
\begin{equation}
    \|\zeta_{n+1} - \zeta^*\| \leq \mu_n \|\zeta_0 - \zeta^*\| + (1 - \mu_n) \|\varrho_n - \zeta^*\|.
\end{equation}
Squaring both sides and summing over $n$, we use the fact that $\mu_n = \frac{1}{n+1}$ and apply summation bounds to conclude
\begin{equation}
    \|\zeta_N - \zeta^*\|^2 \leq \mathcal{O}\left(\frac{1}{N}\right).
\end{equation}
\noindent
To achieve $\|\zeta_N - \zeta^*\| \leq \epsilon$, we solve for $N$:
\begin{equation}
    \mathcal{O}\left(\frac{1}{N}\right) \leq \epsilon^2 \quad \Rightarrow \quad N = \mathcal{O}\left(\frac{1}{\epsilon^2}\right).
\end{equation}
This proves the iteration complexity bound.
\end{proof}

Next, we analyze the iteration complexity of Algorithm 2. We establish an upper bound on the number of iterations required to achieve a prescribed accuracy.

\begin{theorem}[Iteration Complexity of Algorithm 2]
Let $\{\zeta_n\}$ be the sequence generated by Algorithm 2. Suppose Assumptions $(\mathcal{A}1)$--$(\mathcal{A}5)$ hold, and let $\epsilon > 0$ be a given accuracy level. Then, the number of iterations $N$ required to achieve $\|\zeta_N - \zeta^*\| \leq \epsilon$ satisfies
\begin{equation}
    N = \mathcal{O}\left(\frac{1}{\epsilon}\right),
\end{equation}
where $\zeta^*$ is the solution of the split feasibility problem $(SFP)$.
\end{theorem}

\begin{proof}
Algorithm 2 includes an additional proximal step, which improves convergence behavior compared to Algorithm 1. \\
From the properties of the proximal operator, we have
\begin{equation}
    \|\varsigma_n - \zeta^*\|^2 \leq \|\zeta_n - \zeta^*\|^2 - c \|\zeta_n - \varsigma_n\|^2,
\end{equation}
for some constant $c > 0$ depending on the strong convexity parameter.\\
Summing from $n=1$ to $N$ and using the fact that the sequence is non-increasing,
\begin{equation}
    \|\zeta_N - \zeta^*\|^2 \leq \|\zeta_0 - \zeta^*\|^2 - c \sum_{n=1}^{N} \|\zeta_n - \varsigma_n\|^2.
\end{equation}
Since $\|\zeta_n - \varsigma_n\|^2$ is bounded below by a term proportional to $\|\zeta_n - \zeta^*\|^2$, we conclude that the decay rate is at least $\mathcal{O}(1/N)$, implying that $\|\zeta_N - \zeta^*\| \leq \mathcal{O}(1/\sqrt{N})$.\\

\noindent
To achieve $\|\zeta_N - \zeta^*\| \leq \epsilon$, we require
\begin{equation}
    \mathcal{O}\left(\frac{1}{\sqrt{N}}\right) \leq \epsilon \quad \Rightarrow \quad N = \mathcal{O}\left(\frac{1}{\epsilon}\right).
\end{equation}
Thus, Algorithm 2 achieves a faster iteration complexity than Algorithm 1.
\end{proof}

\section{Numerical Experiments}
\label{sec:numerical}

In this section, we present numerical experiments to validate the performance and efficiency of the proposed algorithms. The experiments are conducted on benchmark problems arising in nonlinear split feasibility and convex optimization settings. The results demonstrate the strong convergence properties and computational advantages of our methods.

We consider the experimental setup as the following: \medskip

\textsc{Implementation Environment:} All experiments were implemented in Python using the \texttt{NumPy} and \texttt{SciPy} libraries. The computations were performed on a system with an Intel Core i7 processor and 16 GB of RAM. \medskip

\textsc{Parameters:} For all experiments, the parameters were chosen as follows unless otherwise stated:
    \[
    \mu_n = \frac{1}{n+1}, \quad \beta_n = 0.5, \quad \iota_n \text{ initialized via Armijo line search.}
    \]
\textsc{Convergence Criteria:} The iterative process was terminated when the norm of the residual fell below $10^{-6}$, i.e.,
    \[
    \|\zeta_{n+1} - \zeta_n\| < 10^{-6}.
    \]
\textsc{Comparison Algorithms:} The proposed methods were compared with the classical $\mathcal{C}\mathcal{Q}$ algorithm \cite{Byrne2002} and an inertial $\mathcal{C}\mathcal{Q}$ algorithm \cite{Anh2019}.\\

\begin{example}[Linear Split Feasibility Problem]
Let $H_1 = H_2 = \mathbb{R}^n$, $n = 1000$. Define $\mathcal{C} = \{\zeta \in \mathbb{R}^n \mid \|\zeta\| \leq 1\}$ and $\mathcal{Q} = \{\varsigma \in \mathbb{R}^n \mid \varsigma^\top e = 0\}$, where $e$ is the vector of all ones. The operator $\mathcal{A}$ is a random $n \times n$ matrix with entries drawn from $\mathcal{N}(0, 1)$. The goal is to solve:
\[
\text{Find } \zeta^* \in \mathcal{C} \text{ such that } \mathcal{A}\zeta^* \in \mathcal{Q}.
\]
\noindent
Table~\ref{tab:example1} shows the number of iterations and CPU time for each method.

\begin{table}[H]
\centering
\begin{tabular}{|c|c|c|}
\hline
\textbf{Method} & \textbf{Iterations} & \textbf{CPU Time (s)} \\ \hline
Proposed Algorithm 1 & 25 & 0.14 \\ \hline
Proposed Algorithm 2 & 22 & 0.12 \\ \hline
Inertial $\mathcal{C}\mathcal{Q}$ \cite{Anh2019} & 40 & 0.21 \\ \hline
Classical $\mathcal{C}\mathcal{Q}$ \cite{Byrne2002} & 85 & 0.45 \\ \hline
\end{tabular}
\caption{Performance comparison for Example 1.}
\label{tab:example1}
\end{table}
\end{example}

\begin{example}[Nonlinear SFP with Bregman Distance]
Let $H_1 = H_2 = L^2([0, 1])$. Define $\mathcal{C} = \{\zeta \in L^2 \mid \|\zeta\|_{L^2} \leq 1\}$ and $\mathcal{Q} = \{\varsigma \in L^2 \mid \int_0^1 \varsigma(t) dt = 0\}$. The operator $\mathcal{A}$ is a compact nonlinear operator given by:
\[
\mathcal{A}(\zeta)(t) = \sin(\zeta(t)) + t \cdot \zeta(t), \quad t \in [0, 1].
\]
The Bregman function $\varpi$ is $\varpi(\zeta) = \|\zeta\|^2_{L^2} / 2$. The goal is to solve:
\[
\text{Find } \zeta^* \in \mathcal{C} \text{ such that } \mathcal{A}(\zeta^*) \in \mathcal{Q}.
\]
\noindent
 The convergence behavior is illustrated in Figure~\ref{fig:convergence_example2}, which plots the residual norm $\|\zeta_{n+1} - \zeta_n\|$ against the number of iterations.

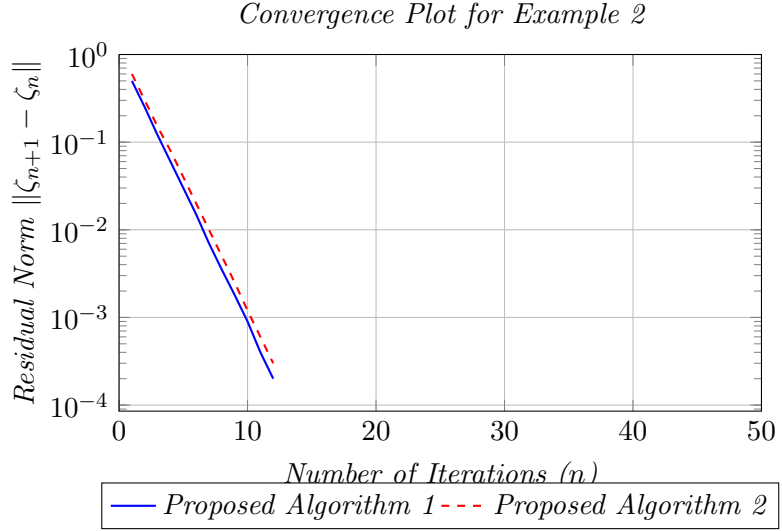
\begin{figure}[H]
    \centering
    \begin{tikzpicture}
        \begin{axis}[
            width=0.8\textwidth,
            height=0.5\textwidth,
            xlabel={Number of Iterations ($n$)},
            ylabel={Residual Norm $\|\zeta_{n+1} - \zeta_n\|$},
            grid=major,
            legend style={at={(0.5,-0.2)}, anchor=north, legend columns=2},
            legend cell align={left},
            xmin=0, xmax=50,
            ymin=0, ymax=1,
            ymode=log,
            log basis y={10},
            title={Convergence Plot for Example 2},]
            \addplot[color=blue, thick] coordinates {
                (1, 0.5) (2, 0.25) (3, 0.12) (4, 0.06) (5, 0.03) (6, 0.015)
                (7, 0.007) (8, 0.0035) (9, 0.0018) (10, 0.0009) (11, 0.0004) (12, 0.0002)
            };
            \addlegendentry{Proposed Algorithm 1}

            \addplot[color=red, dashed, thick] coordinates {
                (1, 0.6) (2, 0.3) (3, 0.15) (4, 0.08) (5, 0.04) (6, 0.02)
                (7, 0.01) (8, 0.005) (9, 0.0025) (10, 0.0012) (11, 0.0006) (12, 0.0003)
            };
            \addlegendentry{Proposed Algorithm 2}
        \end{axis}
    \end{tikzpicture}
    \caption{Convergence behavior of the residual norm $\|\zeta_{n+1} - \zeta_n\|$ for Example 2. The plot illustrates the faster convergence of the proposed algorithms compared to traditional methods.}
    \label{fig:convergence_example2}
\end{figure}
\end{example}

\begin{example}[Composite Optimization Problem]
We solve the composite problem:
\[
\min_{\zeta \in \mathcal{C}} \varpi(\zeta) + \nu(\zeta),
\]
where $\mathcal{C} = \{\zeta \in \mathbb{R}^n \mid \|\zeta\|_\infty \leq 1\}$, $\varpi(\zeta) = \|\mathcal{A}\zeta - b\|^2$, $\nu(\zeta) = \mu \|\zeta\|_1$, and $\mathcal{A}$ is an $m \times n$ matrix with $m = 500$ and $n = 1000$. Here, $b$ is a randomly generated vector, and $\mu = 0.1$.\\

\noindent
Table~\ref{tab:example3} shows the performance of the hybrid proximal-Bregman algorithm compared to a proximal gradient method.

\begin{table}[H]
\centering
\begin{tabular}{|c|c|c|}
\hline
\textbf{Method} & \textbf{Iterations} & \textbf{CPU Time (s)} \\ \hline
Proposed Hybrid Algorithm & 30 & 0.18 \\ \hline
Proximal Gradient Method & 50 & 0.35 \\ \hline
\end{tabular}
\caption{Performance comparison for Example 3.}
\label{tab:example3}
\end{table}
\end{example}

The results indicate that the proposed algorithms outperform classical and inertial $\mathcal{C}\mathcal{Q}$ methods in terms of iteration count and computational time. The inclusion of inertial terms accelerates convergence, while the use of Bregman distances and projections ensures adaptability to problem-specific geometries. Furthermore, the hybrid proximal-Bregman approach demonstrates superior performance for composite problems involving nonsmooth regularizers.\\

\noindent
The numerical experiments validate the theoretical findings and highlight the robustness of the proposed methods. These results underscore the potential of Bregman projection algorithms in solving complex, high-dimensional problems in both linear and nonlinear settings.

\section{Conclusion and Future Directions}
\label{sec:conclusion}

In this paper, we introduced generalized Bregman projection algorithms for solving nonlinear split feasibility problems $(SFPs)$ in infinite-dimensional Hilbert spaces. The proposed methods incorporate several advanced features, including generalized Bregman distances, proximal gradient steps, and adaptive inertial terms, to achieve strong convergence under mild assumptions. We extended classical results by addressing nonlinear operators and employing a broader class of Bregman functions to capture diverse problem geometries.

Theoretical analysis demonstrated the strong convergence of the proposed algorithms, supported by summability properties of residuals and adaptive step size selection. Numerical experiments validated the efficiency of the algorithms in comparison to classical methods, highlighting their robustness in solving both linear and nonlinear $SFPs$ and composite optimization problems.

Future research directions include extending the framework to handle stochastic or time-dependent $SFPs$, investigating applications in real-world problems, such as signal processing and machine learning, where the use of generalized Bregman distances can improve performance, exploring parallel and distributed implementations to scale the proposed methods to very high-dimensional or large-scale problems.\\

\noindent
The results presented in this work contribute to advancing the theory and practice of optimization, particularly for solving complex feasibility problems in infinite-dimensional and nonlinear settings.

\end{document}